\documentclass{article}
\usepackage{graphicx} 
\usepackage{amsmath}
\usepackage{amsthm}
\usepackage{color}
\usepackage{amsfonts}
\usepackage[OT1]{fontenc}
\usepackage{tikz}
\usetikzlibrary{cd}
\usepackage{adjustbox}
\usepackage{url}
\usepackage{enumerate}

\usepackage[hidelinks]{hyperref}

\usepackage[style=numeric,
            backend=bibtex,
            giveninits=true,
            sorting=nyt,
            maxbibnames=99,
            isbn=false, doi=true, eprint=true, url=true]{biblatex}
\numberwithin{equation}{section}

\theoremstyle{plain}
\newtheorem{theorem}{Theorem}[section]
\newtheorem{prp}[theorem]{Proposition}
\newtheorem{cor}[theorem]{Corollary}
\newtheorem{lemma}[theorem]{Lemma}
\newtheorem{conj}{Conjecture}

\theoremstyle{definition}
\newtheorem{example}[theorem]{Example}
\newtheorem{defi}[theorem]{Definition}
\newtheorem{remark}[theorem]{Remark}

\DeclareMathOperator{\Hom}{Hom}
\DeclareMathOperator{\im}{im}
\DeclareMathOperator{\iso}{iso}
\DeclareMathOperator{\Ad}{Ad}
\DeclareMathOperator{\dRtmp}{dR}
\DeclareMathOperator{\Carttmp}{Cart}
\DeclareMathOperator{\const}{const}
\DeclareMathOperator{\id}{id}
\DeclareMathOperator{\Tot}{Tot}
\DeclareMathOperator{\BSStmp}{BSS}
\DeclareMathOperator{\pr}{pr}
\DeclareMathOperator{\End}{End}

\newcommand{\ddR}{d_{\dRtmp}}
\newcommand{\dCart}{d_{\Carttmp}}
\newcommand{\BSS}{\Omega_{\BSStmp}}
\newcommand{\dbl}{/ \! /}

\title{Perspectives on the Bott spectral sequence for Lie groupoids}
\author{Annika Kraasch-Tarnowsky}
\date{July 2026}

\begin{document}

\maketitle

\begin{abstract}
    The cohomology of a differentiable stack can be modelled by the Bott-Shulman-Stasheff bicomplex. However, computing its cohomology is still difficult in many examples. Improving on this requires additional methods such as using spectral sequences and finding models for the cohomology of each column of the bicomplex. This paper provides a first attempt of relating and unifying two approaches of this form. Firstly, Arias Abad has generalised the Bott spectral sequence to Lie groupoids by introducing representations up to homotopy. The second approach is the recent work on a model for the cohomology of differentiable stacks associated to proper and regular Lie groupoids by this author.
\end{abstract}

\section{Introduction}

The Bott spectral sequence, originally appearing in \cite{Bott73} to compute the cohomology groups of $B \mathcal{G}$ for a Lie group $\mathcal{G}$ has been used in various approaches of computing differentiable stack cohomology. This includes research by Getzler \cite{Getz94} on equivariant cohomology of non-compact groups and subsequent work such as by Arias Abad and Uribe \cite{AU15}. In \cite{Abad08} by Arias Abad and his research together with Crainic \cite{AC11, AC13}, a generalisation of the Bott spectral sequence to Lie groupoids was proven. It was constructed to compute the cohomology of the Bott-Shulman-Stasheff bicomplex and its existence has various consequences, for example regarding the van-Est map.

More recently, there have been advances by this author \cite{KT25} in computing the cohomology of differentiable stacks using multiplicative Ehresmann connections as introduced by Fernandes and Marcut \cite{FM23}. These connections induce a notion of invariance using a construction closely related to the concept of jet groupoids, their induced actions and fat groupoids \cite{GSM17}. Underlying the ensuing computation of the cohomology of the Bott-Shulman-Stasheff bicomplex is the same spectral sequence as used by Arias Abad, where the assumptions of properness and regularity allow for a different and very explicit identification of the first page.

The aim of this paper is to review the results of Arias Abad and this author, as well as to explore how they can be related to each other. For this, we will review both constructions with a special emphasis on the regular case of the adjoint representation up to homotopy. It will already be notable from this exposition that both approaches yield terms of the same form for the cohomology of the columns of the Bott-Shulman-Stasheff bicomplex.

We will then study how the concept of a multiplicative Ehresmann connection and related notions such as invariance can be combined with various constructions related to representations up to homotopy. In the main theorem of this paper, this transfer is used to prove that there is an interpretation of the isotropy Cartan model for proper and regular groupoids in terms of representations up to homotopy. Lastly, at the end of this paper there is an explanation on how we may hope to incorporate the well-known Cartan model \cite{Ca50} into this picture as an outlook on ongoing work.

This paper is written in honour of Rui Fernandes' 60th birthday. I would like to thank him for his outstanding support during my PhD project and my visit to Urbana-Champaign in 2024. He majorly influenced my understanding of the $J^E G$-invariant differential forms by suggesting to formulate them using jets and induced actions.

This paper contains results developed within or parallel to my PhD project with the University of Bonn, for which I was a guest and researcher at the Max-Planck-Institute for Mathematics in Bonn. During my time there, I have also been able to discuss with and learn from Camilo Arias Abad, whom I would like to thank for the many insights of which some are presented in this work. This paper has also been influenced by numerous more discussions with a lot of different people, in particular with Christian Blohmann and Sven Holtrop. I would like to thank all of these colleagues for their contributions and the aforementioned institutions for their support.

\section{The coadjoint representation up to homotopy}

To begin, let us review the generalised Bott spectral sequence of Arias Abad. Its first page is formulated in terms of representations up to homotopy.

\begin{defi}[cf. \cite{AC13}, Definition 3.1., Proposition 3.2.]
    Let $G$ be a Lie groupoid. A (split) representation up to homotopy is a graded vector bundle $E^\bullet$ and terms
    \[
    R_i \in \Gamma(G_i, \Hom((p^i_i)^*E^\bullet, (p^0_i))^*E^{\bullet+1-i}))
    \]
    such that the associated derivation on
    \[
    C^\bullet(G, E) := \Gamma(G_\bullet, (p^0_\bullet)^*E)
    \]
    defines a differential. The term $R_0$ can equivalently be expressed as a differential on $E$.
\end{defi}

\begin{remark}
    Note that by the Serre-Swan theorem we can equivalently define a representation up to homotopy not by the datum of a graded vector bundle, but by the datum of the $C^\bullet(G)$-module itself, a viewpoint for example taken in \cite{St19}. This is why, following the convention of \cite{Ob26}, the above definition is called split representation up to homotopy (in contrast to an abstract representation up to homotopy). Alternatively, we may only want to consider the case where we can restrict to a projective and finitely generated $C^\bullet_{N}(G)$-module, denoting the normalised cochain subalgebra. This is equivalent to unitality, the condition that $R_1|_{G_0} = \const_{\id}$ and that
    \[
    R_i \equiv 0
    \]
    on $\bigcup \im(s_j) \subset G_i$ for $i>1$ (cf. \cite{AC13}, Definition 3.6; \cite{St19}, Lemma 3.4.1).
\end{remark}

\begin{defi}[cf. \cite{AC11}, Proposition 3.14.] \label{DefiAdjointRuth}
    Let $G$ be a Lie groupoid and $\mathcal{D}$ a connection on $G_1$ with respect to $s$. Then the adjoint (split two-term) representation up to homotopy is defined on the complex of vector bundles
    \[
    E^\bullet = A[2] \xrightarrow{\rho} T G_0[1]
    \]
    by the terms
    \[
    R^\bullet_1 \in \Gamma(G_1, \Hom(s^*E, t^*E))
    \]
    and
    \[
    R^{-1}_2 \in \Gamma(G_2, \Hom(s^*TG_0, t^*A)).
    \]
    
    $R_1$ is given by
    \begin{align*}
        G_1 \times TG_0 & \rightarrow TG_0 \\
        (x, Y) & \mapsto R^{-1}_1(x)(Y) := T t(Z)
    \end{align*}
    where $Z \in \mathcal{D}_x$ denotes the unique element satisfying $Ts(Z) = Y$ and
    \begin{align*}
        G_1 \times A & \rightarrow A \\
        (x, Y) & \mapsto R^{-2}_1(x)(Y) := T \mu (Z, Y, 0_x)
    \end{align*}
    where $Z \in \mathcal{D}_x$ denotes the unique element satisfying $Ts(Z) = \rho(Y)$.

    $R_2$ is given by the basic curvature of $\mathcal{D}$, i.e the map
    \begin{align*}
        G_2 \times TG_0 & \rightarrow A \\
        (x,y, Y) & \mapsto R^{-1}_2(x)(Y) := T\mu(T \mu(V,W)-Z, 0_{(\mu(x,y))^{-1}}) \in (\ker(Ts))_{1_{t(x)}},
    \end{align*}
    where $W \in \mathcal{D}_y$ denotes the unique element satisfying $Ts(W)=Y$ and $V \in \mathcal{D}_x$ denotes the unique element satisfying $Ts(V)=Tt(W)$, as well as $Z \in \mathcal{D}_{\mu(x,y)}$ denotes the unique element satisfying $Ts(Z)=Y$.
\end{defi}

\begin{remark}
    In the reference, the coadjoint representation is defined on the vector bundle
    \[
        E^\bullet = A \xrightarrow{\rho} T G_0[-1]
    \]
    which corresponds to a shift of the algebra $C(G, E)$ by the constant degree 2. All induced constructions are equivalent up to degree shifts.
\end{remark}

\begin{prp}[cf. \cite{Abad08}, Theorem 5.2.2.]
    Any two choices $\mathcal{D}, \mathcal{D}'$ of connections for the construction of the coadjoint representation up to homotopy induce naturally isomorphic chain complexes.
\end{prp}

One instance where representations up to homotopy appear is the computation of the cohomology of the Bott-Shulman-Stasheff bicomplex.

\begin{defi}[cf. \cite{BSS76}, p. 44]
    The Bott-Shulman-Stasheff bicomplex associated to a simplicial manifold $M_{\bullet}$ (e.g. the nerve $G_\bullet$ of a Lie groupoid $G$) is the bicomplex with vector spaces
    \[
        \BSS^{p,q}(M_{\bullet})=\Omega^p(M_q)
    \]
    as well as the horizontal differential
    \begin{align*}
        d: \BSS^{p,q}(M_{\bullet}) & \rightarrow \BSS^{p+1,q}(M_{\bullet}) \\
        \omega & \mapsto \ddR(\omega)
    \end{align*}
    and the vertical differential
    \begin{align*}
        \delta: \BSS^{p,q}(M_{\bullet}) & \rightarrow \BSS^{p,q+1}(M_{\bullet}) \\
        \omega & \mapsto \sum_{i=0}^{q+1} (-1)^{i+q} (d_i^*).
    \end{align*}
\end{defi}

\begin{remark}
    The sign conventions of this bicomplex and even the choice which index indicates row or column respectively are not standardised.
\end{remark}

This bicomplex is known to compute the cohomology of the differentiable stack associated to a Lie groupoid $G$ (cf. \cite{BX11}, p. 308, 310-311) and therefore a thorough understanding of this object may lead to consequences in the fields of higher differential geometry as well as mathematical physics. As it is difficult to compute in many examples, better methods to determine the cohomology of the Bott-Shulman-Stasheff bicomplex are a topic of ongoing research. Motivation for their possible existence comes from the Cartan model for equivariant cohomology, which actually is a special case of the cohomology of differentiable stacks. The following results by Arias Abad highlight, how representations up to homotopy can be a valuable tool in this setting.

\begin{theorem}[\cite{Abad08}, Theorem 5.3.1.] \label{ThmCoadRuthComputesColumns}
    Let $G$ be a Lie groupoid. Then
    \[
        H^q(\Omega^p(G_{\bullet}), \delta) \cong H^{q+p}(G, S^p(\Ad^*)) = H^{q}(G, S^p(\Ad^*)[p])
    \]
\end{theorem}

\begin{theorem}[\cite{Abad08}, Theorem 5.3.2.] \label{ThmBottSSAriasAbad}
    Let $G$ be a Lie groupoid. There is a spectral sequence converging to the cohomology of BG:
    \[
    \mathcal{E}^{p,q}_1 = H^{p+q}(G, S^q(\Ad^*)) \Rightarrow H^{p+q}(\BSS(G)).
    \]
\end{theorem}

This theorem has various implications, for example on the degrees of the non-vanishing cohomology groups and yields computations for low values of $p$. However, in special cases it becomes significantly more computable. We will in the following set up an example of this in the case of a regular and proper groupoid.

\begin{defi}[cf. \cite{Ob26}, Definition 2.23]
    A strict representation up to homotopy is a complex of representations, i.e. a differential graded vector bundle together with an honest representation of $G$ on each degree, such that the differential is $G$-equivariant. It can be equivalently described as a representation up to homotopy where $R_i=0$ for all $i>1$.
\end{defi}

\begin{example} \label{ExAisoStrictRuth}
    Let $G$ be a regular Lie groupoid, i.e. a Lie groupoid such that its orbit foliation $\mathcal{F} =\im(\rho) \subset TG_0$ is of constant rank. Then the restriction of the map $R^{-2}_1$ of Definition \ref{DefiAdjointRuth} to $A_{\iso} = \ker(\rho)$ can equivalently be written as
    \begin{align*}
        G_1 \times A_{\iso} & \rightarrow A_{\iso} \\
        (x, Y) & \mapsto R^{-2}_1(x)(Y) := T \mu (0_x, Y, 0_x)
    \end{align*}
    and defines a honest $G$-representation. Similarly, the map $R^{-1}_1$ of Definition \ref{DefiAdjointRuth} describing the pseudo-action on $T G_0$ descends to $\nu=T G_0 / \mathcal{F}_0$ as a honest $G$-representation, the normal representation, which is independent of the choice of $\mathcal{D}$ (cf. \cite{dHo12}, Section 3.4).
\end{example}

This example is just one instance of a general phenomenon for regular representations up to homotopy (i.e. where the differential of the graded vector bundle has constant rank in each degree): 

\begin{theorem}[\cite{AC13}, Theorem 3.32.] \label{ThmCohomologyRegularRuths}
    Let $E$ be a unital representation up to homotopy of a Lie groupoid $G$ whose underlying complex is regular. Then each cohomology vector bundle $\mathcal{H}^k(E)$ has the structure of an ordinary representation of $G$ with the action defined by
    \[
        \lambda_g([v]) := [\lambda_g(v)],
    \]
    and there is a spectral sequence:
    \[
        \mathcal{E}_2^{p,q} \cong H^p(G, \mathcal{H}^q(E)) \Rightarrow H^{p+q}(G, E).
    \]
    Moreover, the action on the $\mathcal{H}^k(E)$’s can be extended to a representation up to homotopy structure in the cohomology complex $\mathcal{H}^\bullet(E)$, with zero differential, for which there is a quasi-isomorphism:
    \[
        \Phi: C(G, E) \rightarrow C(G, \mathcal{H}(E)).
    \]
\end{theorem}

In \cite{AC13} it is already noted that this means that if $G$ is proper, the cohomology of both $C(G, \mathcal{H}(E))$ and $C(G, E)$ can be expressed using invariant sections from $G_0$ to $\mathcal{H}(E)$ (cf. \cite{AC13}, Corollary 3.34.). A part of the proof can be generalised from $\mathcal{H}(E)$ to the following setting:

\begin{lemma} \label{LemmaCohomStrictRuths}
Let $G$ be a proper Lie groupoid with a strict representation up to homotopy on the dg-vector bundle $E$. Then
\[
H(G, E) \cong H \left( \ldots \rightarrow \Gamma(G_0, E^0)^G \rightarrow \Gamma(G_0, E^1)^G \rightarrow \Gamma(G_0, E^2)^G \rightarrow \ldots \right)
\]
is an isomorphism of cohomology groups.
\end{lemma}

\begin{proof}
    We can filter the representation up to homotopy by the degree of $E$. As the representation up to homotopy is strict, this filtration is compatible with the differential on $C(G, E)$. If we consider the quotients, we obtain
    \[
    F_q/F_{q-1} = C(G, E^q).
    \]
    Looking at the associated spectral sequence, we have that the first page is therefore given in bidegrees $(p,q)$ by
    \[
    H^p(C, E^q)= \begin{cases}
        0 & \text{if } p \neq 0 \\
        \Gamma(G_0, E^q)^G & \text{otherwise}
    \end{cases}
    \]
    as $G$ is proper (cf. \cite{Cr03}, Proposition 1). For degree reasons, the spectral sequence collapses on the second page, such that the cohomology groups of the first page together with the differential compute the total cohomology. However, the cohomology groups $\Gamma(G_0, E^q)^G$ have a canonical inclusion map into $\Gamma(G_0, E^q) \subset C(G, E)$. The first page differential is induced by the component of the differential on $\Gamma(G_0, E^q)$ with target $\Gamma(G_0, E^{q+1})$. It is explicitly given by applying the differential on $E$ and therefore restricts to the invariant sections due to equivariance (i.e. the strictness assumption).
\end{proof}

\begin{example} \label{ExCohomStrictRuths}
    Let $G$ be a regular groupoid. Consider the strict representation up to homotopy as defined in Example \ref{ExAisoStrictRuth} on $B = A_{\iso} \rightarrow \nu$. Then we can use Lemma \ref{LemmaCohomStrictRuths} to compute
    \[
    H(G_0, B)=\Gamma(G_0, A_{\iso})^G[2] \oplus \Gamma(G_0, \nu)^G[1]
    \]
    as well as
    \[
    H(G, S^p(B^*))= \bigoplus_{a+b=p} \Gamma(G_0, \wedge^a \nu^* \otimes S^b(A^*_{\iso}))^G[-(a+2b)].
    \]
\end{example}

\begin{cor} \label{CorColumnsViaRedRuth}
    Let $G$ be a proper and regular groupoid. Then we have that
    \begin{align*}
        H^q(\Omega^p(G_\bullet), \delta) & \cong H^{q+p}(G, S^p(\Ad^*)) \\ & \cong H^{q+p}(G, S^p(B^*)) \cong \Gamma(G_0, \wedge^{p-q} \nu^* \otimes S^q(A^*_{\iso}))^G.
    \end{align*}
\end{cor}

\begin{proof}
    The first isomorphism is by Theorem \ref{ThmCoadRuthComputesColumns}, the second by Theorem \ref{ThmCohomologyRegularRuths} as $\mathcal{H}(S^p(\Ad^*))=S^p(B^*)$ and the third as in Example \ref{ExCohomStrictRuths}.
\end{proof}

\begin{example} \label{ExTransitiveGrpdBSSViaRuth}
    Consider a proper transitive Lie groupoid $G$. It is automatically regular and satisfies $\nu=0$, such that we obtain
    \[
        H(\Omega^p(G_\bullet), \delta) = \Gamma(G_0, S^p(A^*_{\iso}))^G[-p].
    \]
    This cohomology is only contained in degree $p$ and we immediately obtain that the associated spectral sequence of Theorem \ref{ThmBottSSAriasAbad} collapses at the first page. We may furthermore note that the restriction to a point is an isomorphism on this space, which corresponds to the fact that a transitive groupoid is Morita equivalent to a Lie group.
\end{example}

\section{The isotropy Cartan model}

These results are closely connected to the recent advances on computing differentiable stack cohomology using multiplicative Ehresmann connections.

\begin{theorem}[\cite{KT25}, Theorem 5.2.31] \label{ThmThesisMainThm}
    Let $G$ be a proper and regular groupoid. Then
    \[
    H(\Tot(\BSS(G)))=H( \bigoplus_{a+2b=\bullet} \Gamma(G_0, \wedge^a \nu^* \otimes S^b(A^*_{\iso}))^G, \dCart ).
    \]
\end{theorem}

    We will refer to this model as the isotropy Cartan model. The Cartan differerential on this model is computationally very similar to the differential of the Cartan model, especially on a local frame of $A_{\iso}^*$ compatible with the group structure of the isotropy groupoid. Interestingly, it nonetheless depends not only on the local structure of $A$, but on the global structure of $G$ as well.

    Now, we will first motivate this result before turning to the methods used in the proof.

\begin{example}
    Consider a proper bundle of Lie groups, which implies that the bundle is locally trivial. In the isotropy Cartan model, we have $\nu = T G_0$ and $A_{\iso}=A$, such that on any trivialisation on $U \subset G_0$, the restricted sections
    \[
        \Gamma(U, \wedge (\nu^*) \otimes S(A^*_{\iso}))^G \cong \Omega(U) \otimes \Gamma(U, S(A^*))^G
    \]
    not only agree with the Cartan model, but also can be interpreted as coming from the prouct structure of the stack
    \[
        [ (G|_U)_0 \dbl (G|_U)_1 ] = U \times B\mathcal{G}
    \]
    for the group $\mathcal{G}$ corresponding to the fibres. The differential of the isotropy Cartan model is under the identification $A|_U = \mathfrak{g} \times U$ also given by
    \[
        \ddR \otimes 1: \Omega(U) \otimes (S(\mathfrak{g}^*))^{\mathcal{G}} \rightarrow \Omega(U) \otimes (S(\mathfrak{g}^*))^{\mathcal{G}},
    \]
    which, noting that the differential on the Cartan model of $B \mathcal{G}$ vanishes, further solidifies this point of view.
\end{example}

Note, that even though the structure is locally trivial, interesting global phenomena may appear:

\begin{example}
    Consider the torus bundle over $\mathbb{S}^1$ constructed as follows:
    \[
        \mathbb{T}^2 \times [0,1]/((0,x) \sim (1, \psi(x)) )
    \]
    where $\psi(a, b)=(b^{-1},a) \in \mathbb{S}^1 \times  \mathbb{S}^1$ in terms of the group structure on $\mathbb{S}^1$ or equivalently, $\psi$ corresponds to
    \[
        \begin{pmatrix}
            0 & -1 \\
            1 & 0
        \end{pmatrix} \in \End(\mathbb{R}^2)
    \]
    descending to the quotient. Then (cf. \cite{KT25}, Proposition 4.3.21)
    \[
        H(\Tot(\BSS(G))) = H(\mathbb{S}^1) \otimes \mathbb{R}[x,y,z]/(4z^2-y(x^2-y)).
    \]
    That this expression is slightly involved to compute is a consequence of the fact that $\dCart$ and the cohomology of the Bott-Shulman-Stasheff bicomplex both depend on data exceeding $A$ (note that the algebroid is actually trivial in this example).
\end{example}

\begin{example} \label{ExOrbifoldIsoCartModel}
    Consider a smooth orbifold in the sense of a proper and etalé Lie groupoid $G$. Then the isotropy Cartan model is given by
    \[
        \Gamma(G_0, \wedge (T^* G) )^G \subset \Omega(G_0).
    \]
    This agrees with the notion of orbifold cohomology with coefficients in $\mathbb{R}$ as introduced in \cite{Sa56} and subsequently studied and related to higher structures in \cite{Mo99}.
\end{example}

\begin{example}
    Consider a Lie groupoid $G$ and a Lie subgroupoid $H$ such that
    \begin{enumerate}[a)]
        \item $H_0=G_0$.
        \item $H$ is a bundle of Lie groups.
        \item $H$ contains the connected component containing 0 of the isotropy groupoid.
        \item $H$ is invariant under the adjoint action of $G$ on its isotropy.
    \end{enumerate}
    A possible choice for $H$ is the connected component containing 0 of the isotropy groupoid itself, but also the ridgid isotropy groupoid, which contains all components of the isotropy groupoid that are locally isomorphic to $G_0$ (cf. \cite{KT25}, Definition 5.2.17).

    Then we have that there is a natural map of stacks
    \[
        [H_0 \dbl H_1] \rightarrow [G_0 \dbl G_1],
    \]
    which induces a map on cohomology. On the isotropy Cartan model, this corresponds to the inclusion
    \[
        \Gamma(G_0, \wedge (\nu^*) \otimes S(A^*_{\iso}))^G \subset \Gamma(G_0, \wedge (T^* G_0) \otimes S(A^*_{\iso}))^H
    \]
    where we may identify
    \[
        \Gamma(G_0, \wedge (\nu^*) \otimes S(A^*_{\iso}))^G = \left( \Gamma(G_0, \wedge (\nu^*) \otimes S(A^*_{\iso}))^H \right)^{(G/H)}.
    \]
    In particular, if the isotropy has a manifold structure itself, $G/H$ has no isotropy at all and the invariance precisely identifies the maps that factor through the topological quotient of $G$, which we can interpret as invariance, and vanish along the orbit foliation, which we can interpret as a horizontality condition. So the inclusion of models behaves similarly to the map induced on cohomology by a principal bundle.
    
    This behaviour is well-compatible with the fact that the above map covers the map of stacks
    \[
        [G_0 \dbl G_0] \rightarrow [ G_0 \dbl (G/H)_1],
    \]
    which may be interpreted as a quotient projection, in particular whenever the topological quotient of the orbit foliation is a well-defined manifold. In that case, the quotient manifold is naturally equivalent to $[ G_0 \dbl (G/H)_1]$ and the map indeed corresponds to the quotient projection. In general, as $G/H$ is proper and the isotropy algebroid vanishes (which makes it automatically Morita equivalent to an orbifold), we obtain in analogy to Example \ref{ExOrbifoldIsoCartModel} that the corresponding map on the cohomology is injective, and corresponds to an invariance condition. (As $H$ is a Lie group bundle, we may consider the associated horizontality condition trivially satisfied). If we now interpret $[H_0 \dbl H_1]$ as a pull-back of the diagram
    \begin{center}
        \begin{tikzcd}
            {[ H_0 \dbl H_1 ]} \arrow[r] \arrow[d] & {[ G_0 \dbl G_0 ]} \arrow[d] \\
            {[G_0 \dbl G_1]} \arrow[r] & {[G_0 \dbl (G/H)_1]}
        \end{tikzcd}
    \end{center}
    it is not surprising that some of the properties of the quotient projection, such as injectivity of the induced map on cohomology, are retained.
\end{example}

We will now start to review how the spaces of sections $\Gamma(G, \wedge^a \nu^* \otimes S^b(A^*_{\iso}))^G$ are contextualised to obtain the result of Theorem \ref{ThmThesisMainThm}. For this, we first need to review multiplicative Ehresmann connections. This concept was introduced by Fernandes and Marcut in \cite{FM23} with respect to bundles of ideals, which are subalgebroids of the Lie algebroid with certain properties, in particular the vanishing of the anchor. In our context, we will only need a specific instance of a bundle of ideals, which specifies the definition as follows:

\begin{defi}[cf. \cite{FM23}, Definition 2.4]
    Let $G$ be a regular Lie groupoid. A multiplicative Ehresmann connection with respect to the isotropy algebroid is a distribution $E \subset T G_1$ together with a splitting $TG_1 = E \oplus (\ker(Tt) \cap \ker(Ts))$, such that the map

    \begin{center}
    \begin{tikzcd}
	E \arrow[d, shift right, "Tt" left] \arrow[d, shift left, "Ts" right] \arrow[r, hook] & T G_1 \arrow[d, shift right, "Tt" left] \arrow[d, shift left, "Ts" right] \\
	T G_0 \arrow[r] & T G_0
	\end{tikzcd}
    \end{center}

    is an inclusion of Lie groupoids.
\end{defi}

In the following, whenever we refer to a multiplicative Ehresmann connection, we will mean a multiplicative Ehresmann connection with respect to the isotropy algebroid. Some of the concepts introduced in the following can be defined in a more general setting, but especially the results often use the particular form of the splitting of $T G_1$.

\begin{defi}[\cite{KT25}, Definition 2.1.8]
	The groupoid $J^E G$ is the fat groupoid associated to the VB-groupoid
	
	\begin{center}
	\begin{tikzcd}
	E \arrow[d, shift right, "Tt" left] \arrow[d, shift left, "Ts" right] \arrow[r] & G_1 \arrow[d, shift right, "t" left] \arrow[d, shift left, "s" right] \\
	T G_0 \arrow[r] & G_0
	\end{tikzcd}.
	\end{center}
	
	More explicitly, $(J^E G)_0=G_0$ and
	\[
	(J^E G)_1 = \{ V_x \subset E_x \text{ s.t. } Ts|_{V_x}, Tt|_{V_x} \text{ are isos} \} \subset (J^1 G)_1
	\]
\end{defi}

\begin{remark}
    In the source, the definition of $J^E G$ was done explicitly, but it embeds nicely in the framework of fat groupoids as introduced by Gracia-Saz and Mehta \cite{GSM17}.
\end{remark}

As a subgroupoid of the first jet groupoid, we may immediately transfer some properties:

\begin{lemma}
	If $G$ is a Lie groupoid acting on $p: N \rightarrow G_0$ then there is a canonical induced action of $J^1 G$ (and therefore $J^E G$) on $p \circ \pr_N: TN \rightarrow G_0$.
\end{lemma}

\begin{proof}
The action properties are inherited from the original action in a straightforward way, if $L: G_1 \times^{s, p} N \rightarrow N$ describes the action, then
	\[
	(V_x, Y) \mapsto TL(Z,Y)
	\]
	for $Z \in V_x$ with $Ts(Z)=Tp(Y)$ describes the induced action (cf. \cite{KT25}, Proposition 1.2.10).
\end{proof}

This action now induces a dual action on $T^* N$ and a notion of invariance on $\Omega(N)$.

\begin{example} \label{ExFamilyActionsOnNerve}
The expressions
	\[
	(x, z_1, \dots, z_q) \mapsto(\mu(x, z_1), z_2, \dots, z_q)
	\]
	as well as
	\[
	(x, z_1, \dots, z_q) \mapsto (\dots, z_{i-1}, \mu(z_i, x^{-1}), \mu(x, z_{i+1}), \dots)
	\]
	and
	\[
	(x, z_1, \dots, z_q) \mapsto (\dots, z_{q-1}, \mu(z_q, x^{-1}))
	\]
	define left actions of $G$ on $G_q$. Therefore, we obtain induced actions of $J^E G$ on $T G_q$. We will denote the set of differential $p$-forms jointly invariant under all of these actions by $(\Omega^p(G_q))^{J^E G}$.
\end{example}

It turns out that these invariant forms can be used to compute the first page of the spectral sequence, i.e. the cohomology of the columns of the Bott-Shulman-Stasheff complex.

\begin{theorem}[\cite{KT25}, Theorem 2.3.13]
The inclusion
	\[
	(\Omega^p(G_\bullet))^{J^E G} \subset \Omega^p(G_\bullet)
	\]
	defines the inclusion of a subcomplex with respect to $\delta$ that is also a quasi-isomorphism.
\end{theorem}

This is useful, as the $J^E G$-invariant forms are easily acessible for computation due to the following fact:

\begin{prp}[\cite{KT25}, Proposition 2.1.41]
	The restriction of forms to their values at the unit section yields an isomorphism
	\[
	(\Omega^p(G_q))^{J^E G} \cong \bigoplus_{a+b=p} (\Gamma(G_0, \wedge^a \nu^* \otimes \wedge^b ((A^*_{\iso})^{\oplus q}) ))^G.
	\]
\end{prp}

This does not yet correspond to the expression referred to earlier.

\begin{defi}
	The inclusion
	\[
	S^q(A^*_{\iso}) \subset \wedge^q((A^*_{\iso})^{\oplus q})
	\]
	as the (automatically of multi-degree $(1, \dots, 1)$) elements invariant under the permutation action of $\Sigma_q$ on the summands of $(A^*_{\iso})^{\oplus q}$ induces an inclusion
	\[
	(\Gamma(G_0, \wedge^{p-q} \nu^* \otimes S^q (A^*_{\iso}) ))^G \subset (\Gamma(G_0, \wedge^{p-q} \nu^* \otimes \wedge^q ((A^*_{\iso})^{\oplus q}) ))^G.
	\]
	Elements of $(\Omega^p(G_q))^{J^E G}$ corresponding to the image are referred to as $J^E G$-isotropy Cartan forms.
\end{defi}

\begin{theorem}[\cite{KT25}, Theorem 3.1.20] \label{ThmThesisIsoCartInclusionToColumns}
	The inclusion
	\[
	(\Gamma(G_0, \wedge^{p-\bullet} \nu^* \otimes S^\bullet (A^*_{\iso}) ))^G \subset (\Omega^p(G_\bullet))^G
	\]
	defines the inclusion of a subcomplex with respect to $\delta$ that is also a quasi-isomorphism. Furthermore, the restriction of $\delta$ to the subcomplex vanishes.
\end{theorem}

From this, it is not directly clear how to obtain the result of Theorem \ref{ThmThesisMainThm}. The proof exceeds the scope of this paper and requires a deep analysis of the spectral sequence constructed from the filtration of the Bott-Shulman-Stasheff bicomplex by columns, the choice of homotopies in the quasi-isomorphisms and the impact of the multiplicative Ehresmann connection. In particular, the proof of the collapse of the spectral sequence is explicit and requires the construction of a multiplicative Ehresmann connection which, while it can not be constructed globally flat, is flat on certain restrictions.

Let us also note that a part of the omitted proof is showing that the Cartan differential can equivalently be defined in three ways:

\begin{prp}[\cite{KT25}, Proposition 4.1.20] \label{PrpThesisDefisDCart}
    The following maps
    \[
     (\Gamma(G_0, \wedge^{p-q} \nu^* \otimes S^q (A^*_{\iso}) ))^G \rightarrow (\Gamma(G_0, \wedge^{p+1-q} \nu^* \otimes S^q (A^*_{\iso}) ))^G
    \]
    agree:
    \begin{enumerate}[a)]
        \item The map corresponding to the differential of the first page of the spectral sequence constructed from the filtration of the Bott-Shulman-Stasheff bicomplex by columns.
        \item The map induced by the exteriour covariant derivative on
        \[
        \Gamma(G_0, \wedge^{p-q} \nu^* \otimes S^q (A^*_{\iso}) )
        \]
        with respect to a connection on $A_{\iso}$ obtained by differentiating the restriction of a multiplicative Ehresmann connection to the isotropy groupoid.
        \item The map corresponding on $J^E G$-invariant forms to the component of the de-Rham differential retaining the $E$-vertical degree for any choice of multiplicative Ehresmann connection $E$.
    \end{enumerate}
    The map described this way is denoted by $\dCart$.
\end{prp}

\section{Joining the perspectives}

    In the previous two sections, we have seen two approaches to the cohomology of the Bott-Shulman-Stasheff bicomplex using its filtration by columns. In the cases where both results may be applied, we may hope to compare them.

    \begin{example}
        Consider a transitive and proper Lie groupoid $G$. Then we have already seen in Example \ref{ExTransitiveGrpdBSSViaRuth} that using representations up to homotopy to compute the cohomology of the Bott-Shulman-Stasheff complex yields
        \[
            H(\Tot(\BSS(G)))= (\Gamma(G_0, S^p(A_{\iso}^*)))^G.
        \]
        The isotropy Cartan model describes it as
        \[
            H(\Tot(\BSS(G)))= H \left( (\Gamma(G_0, S^p(A_{\iso}^*)))^G, \dCart \right),
        \]
        but as the Cartan differential is given by the differential of the first page of the associated spectral sequence, it must vanish in analogy to the argument of Example \ref{ExTransitiveGrpdBSSViaRuth}. So both models give the same expression as a result.
    \end{example}
    
    It is a priori not clear if the isomorphisms we have seen in this example agree, even though they look similar. We may note that the similarity of the expressions describing the columns is a general phenomenon, as the same results appear both in the model of Theorem \ref{ThmThesisMainThm} and Corollary \ref{CorColumnsViaRedRuth}. In this section, we aim to compare the two different isomorphisms
	\[
    H(\Omega^p(G_\bullet), \delta) \cong \Gamma(G_0, \wedge^{p-\bullet} \nu^* \otimes S^\bullet(A^*_{\iso}))^G.
    \]
    However, doing so requires some prerequisites.

\begin{lemma} \label{LemmaAdjRuthSplitsforMEC}
    Consider a regular Lie groupoid $G$ with a multiplicative Ehresmann connection $E$. Assume that the connection $\mathcal{D}$ chosen in the construction of the coadjoint action satisfies that $\mathcal{D} \subset E$ as distributions on $G_1$. Then the splitting
    \[
        A = A_{\iso} \oplus (A \cap E).
    \]
    will identify $\Ad$ with the direct sum of the canonical adjoint representation on $A_{\iso}[2]$ and the representation up to homotopy on
    \[
        \mathcal{F}[2] \rightarrow T G_0[1].
    \]
    defined by the inclusion as the differential, the map $R_1^{-1}$ of Definition \ref{DefiAdjointRuth} and its restriction to $\mathcal{F}$ as $R_1^{-2}$ as well as the map $R_2^{-1}$ given by
    \begin{align*}
        G_2 \times TG_0 & \rightarrow \mathcal{F} \\
        (x,y, Y) & \mapsto R^{-1}_2(x)(Y) := Tt(V)-Tt(Z) \in T_{t(x)} G_0,
    \end{align*}
    where for $W \in \mathcal{D}_y$ the unique element satisfying $Ts(W)=Y$ we denote by $V \in \mathcal{D}_x$ the unique element satisfying $Ts(V)=Tt(W)$, and by $Z \in \mathcal{D}_{\mu(x,y)}$ we denote the unique element satisfying $Ts(Z)=Y$.
\end{lemma}

\begin{proof}
    The differential on $C(G, \Ad)$ decomposes into three components corresponding to $R'_0, R'_1$ and $R'_2$ in the construction of $\Ad$. $R'_0$ corresponds to the differential of the graded vector bundle, which vanishes on $A_{\iso}$. On $A \cap E$, the differential is given by $\rho$, and as $\rho$ precisely corresponds to the isomorphism $A \cap E \cong \mathcal{F}$, it becomes the inclusion under this identification. So we obtain the same data as from the differential $R_0$ as described in the statement.
    
    The component of the differential corresponding to $R'_1$ comes from a sum where each summand corresponds to a face map of the nerve of the groupoid. All but one of the summands retain any splitting by definition. Futhermore, the remaining summand corresponding to $d_0$ is given by applying the values of $R'_1$ depending on the point on which we evaluate. However, the pseudo-action on $A$ defined by $\mathcal{D}$ restricts to an actual action on $A_{\iso}$ and retains $E$-horizontality by the multiplicativity of $E$. Therefore, it also retains $A \cap E$ and we indeed have that all summands constituting the differential respect the splitting. We can additionally check that
    \[
        \rho((R'_1)^{-2}(x)(Y))= Tt(T \mu (Z, Y, 0_x)) = Tt(Z),
    \]
    where $Z \in \mathcal{D}_x$ denotes the unique element satisfying $Ts(Z)=\rho(Y)$, such that we obtain
    \[
        \rho((R'_1)^{-2}(x)(Y))= (R'_1)^{-1}(x)(\rho(Y)).
    \]
    
    Lastly, the component of the differential coming from $R'_2$ measures the failure of $\mathcal{D}$ to be multiplicative and define a honest representation. But as $\mathcal{D} \subset E$ we have that multiplying two $\mathcal{D}$-horizontal sections of $Ts$ will still yield an $E$-horizontal section of $Ts$. Therefore, they can only differ by elements of $\ker(Ts) \cap \nolinebreak E$, which correspond to $A \cap E$. More precisely, in the formulas of Definition \ref{DefiAdjointRuth}, we have that both $T\mu (V,W), Z \in Ts^{-1}(Y) \cap E_{\mu(x,y)} $ and therefore
    \[
        (R'_2)^{-1}(x)(Y) \in (\ker(Ts))_{1_{t(x)}} \cap E = A_{t(x)} \cap E.
    \]
    We can explicitly check
    \begin{align*}
        \rho((R'_2)^{-1}(x)(Y)) & = Tt(T\mu(T \mu(V,W)-Z, 0_{(\mu(x,y))^{-1}})) \\ & = Tt(T \mu(V,W)-Z) = Tt(V) - Tt(Z)
    \end{align*}
    to obtain the map as described.
\end{proof}

\begin{lemma} \label{LemmaProjToNormalRepFromExtension}
    The projection map
    \[
     \left( \mathcal{F}[2] \rightarrow T G_0[1] \right) \rightarrow \nu[1]
    \]
    is a map of representations up to homotopy for the structure of Lemma \ref{LemmaAdjRuthSplitsforMEC} and the normal representation on $\nu$.
\end{lemma}

\begin{proof}
    The normal representation is defined precisely by the pseudo-action corresponding to $R^{-1}_1$, which projects to a honest $G$-action on $\nu$ (cf. \cite{dHo12}, Section 3.4).
\end{proof}

\begin{cor} \label{CorInclusionDualResAdIntoDualAd}
    The inclusion $B^* \subset \Ad^*$ induced by the splitting of $A$ associated to a multiplicative Ehresmann connection induces an inclusion of representations up to homotopy.
\end{cor}
    
\begin{lemma} \label{LemmaNaturalEquivarPullbackLinAction}
	Consider a vector bundle $b: V \rightarrow M$. Furthermore consider a fibrewise linear action $L$ of a Lie groupoid $G$ on $V$ that covers an action $\underline{L}$ of $G$ on $p: M \rightarrow G_0$ and an action $K$ of $G$ on $p': N \rightarrow G_0$. Now consider an equivariant map $f: N \rightarrow M$. Then there is a natural linear action of $G$ on $f^* V$ covering the action on $N$ and the map
	\[
	f^*: \Gamma(M, V) \rightarrow \Gamma(N, f^* V)
	\]
	restricts to the equivariant sections
	\[
	f^*: \Gamma(M, V)^G \rightarrow \Gamma(N, f^* V)^G.
	\]
\end{lemma}

\begin{proof}
	The action of $G$ on $f^* V$ is straightforwardly defined by
	\[
	(g, (n, X)) \mapsto (K(g, n), L(g, X))
	\]
	where by construction we have that $s(g)=p'(n)=p(f(n))$ and $f(n)=\pr_M(X)$, such that the terms are well-defined and furthermore $f(K(g,n))=\underline{L}(g, f(n))$, such that they indeed describe an element of $f^* V$. The action properties are immediately inherited and naturality can be seen from the expressions. It is also clear that a section $\Gamma(N, f^* V)$ is equivariant if and only if the associated map from $N$ to $V$ is, such that the last statement follows.
\end{proof}

\begin{example}
	Consider any vector bundle $V \rightarrow G_0$ with a linear action covering the unique action of $G$ on $\id: G_0 \rightarrow G_0$. We may choose this to be coadjoint action of a Lie groupoid $G$ on $S^b(A^*_{\iso})$ or the conormal action of $G$ on $\wedge^a \nu^*$. Furthermore consider the map
	\[
	p^0_q: G_q \rightarrow G_0
	\]
	mapping an element to its left-most target. This is equivariant with respect to the action of $G$ on $G_q$ by left multiplication, such that we obtain the map
	\[
	(p^0_q)^* : \Gamma(G_0, V)^G \rightarrow \Gamma(G_q, (p^0_q)^* V)^G
	\]
	for any such $V$.
	
	Furthermore, for the induced actions of $G$ on the $(p^0_q)^* V$ we can note that any face map $d_i: G_{q+1} \rightarrow G_q$ with $i>0$ is actually equivariant with respect to the underlying left multiplication, yielding the map
	\[
	d_i^*: \Gamma(G_q, (p^0_q)^*V)^G \rightarrow \Gamma(G_{q+1}, (d_i)^*(p^0_q)^* V)^G.
	\]
	As we can factor
	\[
	p^0_{q+1} = p^0_{q} \circ d_i
	\]
	for any face map with $i>0$, we can identify $(d_i)^*(p^0_q)^* V=(p^0_{q+1})^* V$. That the $G$-actions on these vector bundles agree follows from naturality.
\end{example}

\begin{lemma} \label{LemmaNaturalInvarPullbackLinAction}
	Consider a vector bundle $b: V \rightarrow M$. Furthermore consider an action $K$ of $G$ on $p': N \rightarrow G_0$. Now consider an invariant map $f: N \rightarrow M$. Then there is a natural linear action of $G$ on $f^* V$ covering the action on $N$ and the map
	\[
	f^* : \Gamma(M, V) \rightarrow \Gamma(N, f^* V)
	\]
	has its image contained in $(\Gamma(N, f^* V))^G$.
\end{lemma}

\begin{proof}
	The action of $G$ on $f^* V$ is straightforwardly defined by
	\[
	(g, (n, X)) \mapsto (K(g, n), X)
	\]
	where by construction we have that $f(K(g,n))=f(n)=\pr_M(X)$, such that the expression indeed defines an element of $f^* V$. The action properties are immediately inherited and naturality can be seen from the expressions. It is also clear that an element of $\Gamma(N, f^* V)$ is equivariant only if it corresponds to an invariant map from $N$ to $V$ and that precomposition maps $\Gamma(M, V)$ to precisely that space.
\end{proof}

\begin{example}
	Consider any vector bundle $V \rightarrow G_0$. We may choose this to be coadjoint action of a Lie groupoid $G$ on $S^b(A^*_{\iso})$ or the conormal action of $G$ on $\wedge^a \nu^*$. Furthermore consider the map
	\[
	p^0_q : G_q \rightarrow G_0
	\]
	mapping an element to its left-most target. This is invariant with respect to any of the actions of $G$ on $G_q$ as in Example \ref{ExFamilyActionsOnNerve} with the exception of left multiplication. Therefore, we get a family of actions of $G$ on $(p^0_q)^* V$.
\end{example}

\begin{theorem} \label{ThmMainTwoIsosAgree}
    Let $G$ be a regular and proper Lie groupoid. The two isomorpisms
    \[
    H(\Omega^p(G_\bullet), \delta) \cong \Gamma(G_0, \wedge^{p-\bullet} \nu^* \otimes S^\bullet(A^*_{\iso}))^G
    \]
    given by Corollary \ref{CorColumnsViaRedRuth} and Theorem \ref{ThmThesisIsoCartInclusionToColumns} agree up to scaling in each degree by $\pm 1$.
\end{theorem}

\begin{proof}
    Let us recall the proof of \cite{AC13} on the isomorphism
    \[
    H(\Omega^p(G_\bullet), \delta) \cong H(G, S^p(\Ad^*)).
    \]
    It is given by an augmented bicomplex as follows:

    \begin{center}
    \adjustbox{scale=0.8,center}{
        \begin{tikzcd}
             & \vdots & \vdots & \vdots &  \\
            0 \arrow[r] & C(G, S^p(\Ad^*)[p])^2 \arrow[r, "\flat^*_0"], \arrow[u, "D"] & C(G^{(0)}, S^p(\Ad^*)[p])^2 \arrow[r, "\flat^*"] \arrow[u, "D"] & C(G^{(0)}, S^p(\Ad^*)[p])^2 \arrow[r, "\flat^*"] \arrow[u, "D"] & \dots \\
            0 \arrow[r] & C(G, S^p(\Ad^*)[p])^1 \arrow[r, "\flat^*_0"], \arrow[u, "D"] & C(G^{(0)}, S^p(\Ad^*)[p])^1 \arrow[r, "\flat^*"] \arrow[u, "D"] & C(G^{(0)}, S^p(\Ad^*)[p])^1 \arrow[r, "\flat^*"] \arrow[u, "D"] & \dots \\
            0 \arrow[r] & C(G, S^p(\Ad^*)[p])^0 \arrow[r, "\flat^*_0"], \arrow[u, "D"] & C(G^{(0)}, S^p(\Ad^*)[p])^0 \arrow[r, "\flat^*"] \arrow[u, "D"] & C(G^{(0)}, S^p(\Ad^*)[p])^0 \arrow[r, "\flat^*"] \arrow[u, "D"] & \dots \\
             &  & \Omega^p(G_0) \arrow[r, "\delta"] \arrow[u, "d_0^*"] & \Omega^p(G_1) \arrow[r, "\delta"] \arrow[u, "d_0^*"] & \dots \\
             &  & 0 \arrow[u] & 0 \arrow[u] &  
        \end{tikzcd}
    }
    \end{center}

    where the groupoids $G^{(i)}$ are constructed as the action groupoids of the action of $G$ on $G_{i+1}$. We can identify
    \[
        G^{(i)}_n = G_{i+n+1}.
    \]
    The horizontal arrows $\flat^*$ are given by a sum of maps corresponding to the face maps $d_j: G_{n+2} \rightarrow G_{n+1}$ for $j>0$, which are equivariant with respect to the action of $G$. The associated maps $G^{(i)}_{n+1} \rightarrow G^{(i)}_n$ on the nerve correspond to the face maps $d_{j+i}$. The vertical differential $D$ in the bicomplex is the differential of the representation up to homotopy.

    Now, we aim to construct an invariant subcomplex of the given complex, such that its horizontal augmenting complex is given by the $J^E G$-isotropy Cartan forms and its vertical augmenting complex is of the form
    \[
        (\Gamma(G_\bullet, (p^0_q)^*S^p(\Ad^*)))^{J^E G} \subset \Gamma(G_\bullet, (p^0_q)^*S^p(\Ad^*)) = C^\bullet(G, S^p(\Ad^*)).
    \]
    Let us now make sense of this notion:

    We have that the action of $J^1 G$ on the vector spaces of $\Ad$ and $\Ad^*$ (by (co)adjoint action and trivially induced action) defines an action on the graded vector bundle of the representation up to homotopy covering the trivial action on $G_0$, as $\rho$ is equivariant. This will also induce an action on the underlying graded vector bundle of $S^p(\Ad^*)$ compatible with its differential. If we restrict this action along the inclusion, we can make sense of the expression
    \[
        (\Gamma(G_0, S^p(\Ad^*)))^{J^E G}.
    \]
    
    For $G_q$ with $q>0$, we aim to construct a family of $q+1$ actions of $J^1 G$ on $(p^0_q)^* S^p(\Ad^*)$ covering the family of actions of $G$ on some $G_q$ as described in Example \ref{ExFamilyActionsOnNerve} (which we can extend trivially to actions of $J^1 G$ on $G_q$). However, as $p^0_q$ is either equivariant or invariant with respect to these actions, we obtain the required data by Lemma \ref{LemmaNaturalEquivarPullbackLinAction} and \ref{LemmaNaturalInvarPullbackLinAction}. We will restrict the actions along the inclusion $J^E G \subset J^1 G$, which yields the well-defined inclusion
    \[
    (\Gamma(G_\bullet, (p^0_q)^*S^p(\Ad^*)))^{J^E G} \subset \Gamma(G_\bullet, (p^0_q)^*S^p(\Ad^*)) = C^\bullet(G, S^p(\Ad^*)).
    \]
    Here, the invariance denotes invariance under the entire family of actions. Note that invariance under $J^E G$ implies invariance under its closure
    \[
    \overline{(J^E G)_1} = \{ V_x \subset E_x \text{ s.t. } Ts|_{V_x} \text{ is iso} \}
    \]
    (which is a well-defined manifold, cf. \cite{KT25}, Lemma 2.1.5). By considering these generalised jets at units, we can note that for any $X \in A \cap E \subset A$ as well as every $Y \in \mathcal{F} \subset TG_0$, there is a generalised jet annulling this element. We can therefore see that
    \[
    (\Gamma(G_q, (p^0_q)^*S^p(\Ad^*)))^{J^E G} = (\Gamma(G_q, (p^0_q)^*S^p(B^*)))^{J^E G} = (\Gamma(G_q, (p^0_q)^*S^p(B^*)))^{G}
    \]
    where the inclusion of $B^*$ into $\Ad^*$ is along the splitting
    \[
    A = A_{\iso} \oplus (A \cap E)
    \]
    and the fact that the action of $J^E G$ (and even $J^1 G$) on $B$ descends to $G$ follows directly from its definition. Furthermore, we can simplify
    \[
        (\Gamma(G_\bullet, (p^0_q)^*S^p(B^*)))^{G} = (\Gamma(G_0, S^p(B^*)))^{G}
    \]
    by restricting to the values on the unit section, as these uniquely define the values at all points implicitly. The invariance under the family of actions on the left transforms to the coadjoint action on the right, as the family of (pairwise commutative) $G$-actions on the left hand side can only map an element on the unit section to the unit section by acting with the same element on the left, right and all diagonals, which corresponds to the coadjoint action by construction.
    
    We can similarly proceed for the groupoids $G^{(k)}$. Here, for each $G^{(k)}_q$ we aim to construct a family of $k+q+2$ actions, covering the actions of Example \ref{ExFamilyActionsOnNerve} using the identification $G^{(k)}_q = G_{k+q+1}$. Note, that as $G^{(k)}$ is an action groupoid of $G$, its algebroid $A^{(k)}$ is a pull-back of $A$ along the map $G_{k+1} \rightarrow G_0$ along which $G$ acts, which is $p^0_{k+1}$. So actually, we obtain that
    \[
        (p^0_q)^* A^{(k)} = (p^0_{q+k+1})^* A.
    \]
    This means that we can use the same actions as constructed before to define the actions on $A^{(k)}$ and its pull-backs.
    
    It remains to construct the corresponding actions on the bundles $(p^0_q)^* T G^{(k)}_0$. On $G^{(k)}_0$, we have that by the induced action property, all of the actions of $G$ on $G_{k+1}$ induce an action of $J^1 G$ on $T G^{(k)}_0$. We can check that for all of the actions of $G$ on $G_{k+q+1}$, the map to $G_{k+1}$ corresponding to $p^0_q$ in $G^{(k)}$ is either equivariant or invariant with respect to one of the actions on $G_{k+1}$, allowing us to construct an induced action of $J^1 G$ on $(p^0_q)^* T G^{(k)}_0$. As pull-backs and symmetric powers commute, we have now well-defined actions on the vector bundles $(p^0_q)^* S^p(\Ad^*) \rightarrow G^{(k)}_q$ covering the $k+q+2$ actions of $G$ on $G_{k+q+1}$.

    If we analyse again the action of $\overline{J^E G}$ on $\Gamma(G^{(k)}_\bullet, (p^0_q)^*S^p(\Ad^*))$, in particular the action of the jets at the units, we can see that elements of $(p^0_k)^* (A \cap E) $ over $G^{(k)}_0$ and their corresponding elements in the vector bundle over $G^{(k)}_q$ get anulled. This also happens for those elements of $(Tp^i_k)^{-1}(\mathcal{F}) \subset T G^{(k)}_0$ (this expression is independent of $i$, cf. \cite{KT25}, Proposition 1.1.16) which are $E$-horizontal under all projections $G^{(0)}_k = G_{k+1} \rightarrow G_1$. Therefore,
    \[
    (\Gamma(G^{(k)}_q, (p^0_q)^*S^p(\Ad^*)))^{J^E G} = (\Gamma(G^{(k)}_q, (p^0_q)^*S^p((B^{(k)})^*)))^{J^E G},
    \]
    where
    \[
        B^{(k)} = \left( (p^0_{k+1})^* A_{\iso}[2] \rightarrow ((p^0_k)^*\nu \oplus (p^0_{k+1})^* (A_{\iso}) \oplus \dots \oplus (p^{k}_{k+1})^* (A_{\iso}))[1]  \right)
    \]
    and the projection from $\Ad$ to $B^{(k)}$ is the one induced by the quotient map dividing out the anulled subspaces. The structure of a direct sum in degree -1 comes from the data of the multiplicative Ehresmann connection describing a lift of the normal bundle on $G_0$ to $T G^{(k)}_0$ up to elements of $(Tp^i_k)^{-1}(\mathcal{F})$ which satisfy the $E$-horizontality condition. In the quotient of $T G^{(k)}_0$ by the vanishing set this lift determines a complement of the joint kernel of the $Tp^i_q$ yielding a direct sum decomposition and moreover the projections from $G_{k+1}$ to the different copies of $G_1$ split that joint kernel further into the $k+1$ components. The differential of $B^{(k)}$ is given by the identity on the component $(p^0_{k+1})^* A_{\iso}$. The induced representation up to homotopy on $B^{(k)}$ is strict and corresponds (using the action groupoid structure) to the pull-back of adjoint representation on $A_{\iso}$ and the normal representation on $\nu$ along $p^0_{k+1}$ as well as the trivial representation on the other copies of $A_{\iso}$.

    Note, that this means that we can also interpret $B^{(k)}$ as the direct sum of
    \[
        \left( (p^0_{k+1})^* A_{\iso}[2] \rightarrow (p^0_{k+1})^* A_{\iso}[1] \right) \oplus ((p^0_k)^*\nu \oplus \dots \oplus (p^{k}_{k+1})^* (A_{\iso}))[1].
    \]
    
    Again, the actions of $J^E G$ descend to actions of $G$ and we obtain
    \begin{align*}
        (\Gamma(G^{(k)}_q, (p^0_q)^*S^p((B^{(k)})^*)))^{J^E G}& =(\Gamma(G^{(k)}_q, (p^0_q)^*S^p((B^{(k)})^*)))^{G} \\ & = (\Gamma(G_0, S^p(\nu^* \oplus (A^*_{\iso})^{\oplus k+1} \rightarrow A^*_{\iso})))^{G}.
    \end{align*}
    
    We can note that all horizontal arrows in the augmented bicomplex of \cite{AC13} as above retain the invariant subcomplexes. This is because the maps $\flat$ are comprised of sums of maps corresponding to face maps, which are each either equivariant or invariant with respect to any action on the target, such that Lemma \ref{LemmaNaturalEquivarPullbackLinAction} and Lemma \ref{LemmaNaturalInvarPullbackLinAction} apply. The vertical differentials also retain the invariant subcomplexes. This is because they are in the image of $C(G^{(k)}, (B^{(k)})^*)$ and therefore in analogy to Corollary \ref{CorInclusionDualResAdIntoDualAd}, we can compute the differential by computing it on the smaller subcomplex. However, the differential on its underlying graded vector bundle is trivially equivariant with respect to all actions and as the representation up to homotopy is strict we can use that it otherwise only consists of a sum of either face maps or the action map of the representation. These, as well, are maps that are each either equivariant or invariant with respect to any action on the domain.
    
    Therefore, we obtain the subcomplex of the augmented bicomplex starting as follows

    \begin{center}
    \adjustbox{scale=0.8,center}{
        \begin{tikzcd}[column sep = tiny]
             & \vdots & \vdots  \\
            0 \arrow[r] & (\Gamma(G_0, \wedge^{p-2}(\nu^*) \otimes S^2(A^*_{\iso})))^{G} \oplus \dots \oplus \cdots \arrow[r, "\flat^*_0"], \arrow[u, "D"] & (\Gamma(G_0, \wedge^{p-2}(\nu^* \oplus A^*_{\iso}) \otimes S^2(A^*_{\iso})))^{G} \oplus \cdots \arrow[u, "D"] \\
            0 \arrow[r] & (\Gamma(G_0, \wedge^{p-1}(\nu^*) \otimes A^*_{\iso}))^{G} \oplus (\Gamma(G_0, \wedge^p(\nu^*)))^{G} \arrow[r, "\flat^*_0"] \arrow[u, "D"] & (\Gamma(G_0, \wedge^{p-1}(\nu^* \oplus A^*_{\iso}) \otimes A^*_{\iso}))^{G} \oplus \cdots \arrow[u, "D"] \\
            0 \arrow[r] & (\Gamma(G_0, \wedge^p(\nu^*)))^{G} \arrow[r, "\flat^*_0"], \arrow[u, "D"] & (\Gamma(G_0, \wedge^p(\nu^* \oplus A^*_{\iso})))^{G} \arrow[u, "D"] \\
             &  & \Omega^p(G_0) \arrow[u, "d_0^*"]  \\
             &  & 0 \arrow[u] 
        \end{tikzcd}
    }
    \end{center}
    with columns in general given by

    \begin{center}
    \adjustbox{scale=0.8,center}{
        \begin{tikzcd}[column sep = small]
            & \vdots &  \\
            \dots \arrow[r, "\flat^*"] & (\Gamma(G_0, \wedge^{p-2}(\nu^* \oplus (A^*_{\iso})^{\oplus k+1}) \otimes S^2(A^*_{\iso})))^{G} \oplus \dots \oplus \dots \arrow[r, "\flat^*"] \arrow[u, "D"] & \dots \\
            \dots \arrow[r, "\flat^*"] & (\Gamma(G_0, \wedge^{p-1}(\nu^* \oplus (A^*_{\iso})^{\oplus k+1}) \otimes A^*_{\iso}))^{G} \oplus (\Gamma(G_0, \wedge^p(\nu^* \oplus (A^*_{\iso})^{\oplus k+1})))^{G} \arrow[r, "\flat^*"] \arrow[u, "D"] & \dots \\
            \dots \arrow[r, "\flat^*"] & (\Gamma(G_0, \wedge^p(\nu^* \oplus (A^*_{\iso})^{\oplus k+1})))^{G} \arrow[r, "\flat^*"] \arrow[u, "D"] & \dots \\
             \dots \arrow[r, "\delta"] & \Omega^p(G_k) \arrow[r, "\delta"] \arrow[u, "d_0^*"] & \dots \\
              & 0 \arrow[u] &  
        \end{tikzcd}
    }
    \end{center}
    where each component of the direct sum in each degree is repeated in the direct sum of the next degree, where the first time a component appears, it corresponds to a section $G^{(k)}_0$, while in higher degrees it corresponds to a section on $G^{(k)}_i$.

    In particular, all vector spaces corresponding to a fixed total degree $q$ have a subcomplex (contained in the part corresponding to sections of $G^{(k)}_0$) of the form
    \[
        (\Gamma(G_0, \wedge^{p-q}(\nu^*) \otimes S^q(A^*_{\iso})))^G,
    \]
    where we use the inclusion
    \[
        S^q(A^*_{\iso}) \subset S^{k+1}(A^*_{\iso}) \otimes S^{q-k-1}(A^*_{\iso})
    \]
    as well as
    \[
        S^{k+1}(A^*_{\iso}) \subset \wedge^{k+1}((A^*_{\iso})^{\oplus k+1}).
    \]

    Now let us analyse the differentials. As we are considering the differential of a strict representation up to homotopy, the vertical differential $D$ decomposes into $D^0$, induced by the anchor map, and $D^1$ given by a sum corresponding to face maps and the action map. As long as we are considering invariant sections, all of the components of the sum correspond to the identity on each $(\Gamma(G_0, \wedge^{a}(\nu^* \oplus (A^*_{\iso})^{\oplus k+1}) \otimes S^b(A^*_{\iso})))^{G}$. As the sum is alternating, we will therefore obtain that $D^1=0$, when it corresponds to some map of the form $\Gamma(G^{(k)}_q, (p^0_q)^*S^p((B^{(k)})^*)) \rightarrow \Gamma(G^{(k)}_{q+1}, (p^0_{q+1})^*S^p((B^{(k)})^*))$ for $q$ even. When $q$ is odd, we will obtain $D^1= \id$. The $D^0$ component will be induced by mapping $A_{\iso}$ to the first copy in the direct sum $(A_{\iso})^{\oplus k+1}$ and in particular will map elements corresponding to
    \[
        (\Gamma(G_0, \wedge^{p-q}(\nu^*) \otimes S^q(A^*_{\iso})))^G,
    \]
    under the identification $S^q(A^*_{\iso}) \subset S^{k+1}(A^*_{\iso}) \otimes S^{q-k-1}(A^*_{\iso})$ to their corresponding expressions under the identification $S^q(A^*_{\iso}) \subset S^{k}(A^*_{\iso}) \otimes S^{q-k}(A^*_{\iso})$, where $S^{k}(A^*_{\iso}) \subset \wedge^{k}((A^*_{\iso})^{\oplus k+1})$ as elements of degree $0$ in the first copy and degree $1$ in all others invariant under any permutation that fixes the first copy.

    Now let us consider the horizontal differential. Firstly, we can check that each of the summands of $\flat$ corresponds to the identity on the degree -2 part (i.e. $(p^0_k)^*A_{\iso}[2]$) of the $B^{(k)}$. Therefore it is sufficient to analyse $\flat^*$ and its components at the unit sections on the expressions $\Gamma(G_0, \wedge^p(\nu^* \oplus (A_{\iso}^*)^{\oplus k+1}))$. Note that as the face maps defining $\flat^*$ restrict to the unit sections, the invariance is not necessary to obtain a well-defined map on the values at the unit section under the pull-back along face maps. The components of $\flat^*_i$ correspond to the family of face maps $d_i: G_{k+2} \rightarrow G_{k+1}$ with $i>0$ and the pull-back maps are added alternatingly. In \cite{KT25} (cf. Lemma 3.1.11; Corollary 3.1.12), there is an explicit analysis of the face maps on these expressions showing in particular that the alternating sum of all face maps (including $(d_0)^*$) vanishes on expressions of the form
    \[
        \Gamma(G_0, \wedge^{a}(\nu^*) \otimes S^{k+1}(A^*_{\iso})) \subset \Gamma(G_0, \wedge^{a}(\nu^*) \otimes \wedge^{k+1}((A^*_{\iso})^{\oplus k+1})).
    \]
    Furthermore, we can explicitly compute that $d_0$ corresponds to the inclusion of $(A^*_{\iso})^{k+1}$ into $(A^*_{\iso})^{k+2}$ by extending by $0$ in the first copy of $A^*_{\iso}$. By the choice of signs in \cite{AC13}, we find that $\flat^*$ corresponds to
    \[
        \sum_{i=1}^{k+2} (-1)^{i+1} d_i^* = d_0^* - \sum_{i=0}^{k+2} (-1)^{i} d_i^*
    \]
    and as the map induced by the second summand vanishes, $\flat^*$ corresponds to extension by 0 in the first summand of $(A^*_{\iso})^{k+2}$ as well. The same argument holds for $\flat^*_0$, which is analogously constructed.

    One consequence of this is, that for any $J^E G$-isotropy Cartan form with values at the unit section given by
    \[
        \alpha \in (\Gamma(G_0, \wedge^{a}(\nu^*) \otimes S^b(A^*_{\iso})))^G
    \]
    the associated elements
    \[
        \alpha_k \in (\Gamma(G_0, \wedge^{a+k+1}(\nu^* \oplus (A^*_{\iso})^{\oplus k+1}) \otimes S^{b-k-1}(A^*_{\iso})))^{G}
    \]
    for $-1 \leq k \leq b-1$ satisfy that
    \[
        D(\alpha_{k+1}) = \flat(\alpha_k),
    \]
    as well as
    \[
        (d_0)^*(\alpha)= \flat(\alpha_{b-1}).
    \]
    
    The correspondence between the two augmenting complexes is now obtained by considering the quasi-isomorphisms between their cohomology and the cohomology of the associated total complex of the entire structure. Note that in \cite{AC13}, the notion of augmented bicomplex used requires commuting differentials. Therefore, we must choose a convention on which differentials to multiply by $-1$ to construct anti-commuting differentials. All of the complexes constructed this way are isomorphic, with the isomorphisms on each vector space given by multiplication with $\pm 1$. Therefore, when iteratively constructing a closed element of the total complex corresponding to $\alpha$, we will obtain
    \[
        \alpha + \sum_{k=-1}^{b-1} \pm \alpha_k,
    \]
    where the sign of $\alpha_{-1}$ will only depend on $b$ (and the choice of convention). Therefore, indeed, the cohomology class $[\alpha]$ represented by a unique $J^E G$-isotropy Cartan form $\alpha$ will under the isomorphism of \ref{ThmCoadRuthComputesColumns} correspond to the cohomology class $[\alpha_{-1}]$, where $\alpha$ and $\alpha_1$ correspond to the same expression up to scaling by $\pm 1$ when abstractly written as elements
    \[
        \alpha, \alpha_{-1} \in (\Gamma(G_0, \wedge^{a}(\nu^*) \otimes S^{b}(A^*_{\iso})))^{G}.
    \]
    
    In the last step of this proof, we now need to consider the isomorphism
    \[
    H(G, S^p(\Ad^*)) \cong H(G, \mathcal{H}(S^p(\Ad^*)))= H(G, S^p(B^*))
    \]
    that we used in Corollary \ref{CorColumnsViaRedRuth}. It is defined using the perturbation Lemma. We will first study what happens when the perturbation lemma is applied for the projection induced by the splitting
    \[
    A = A_{\iso} \oplus (A \cap E) \cong A_{\iso} \oplus \mathcal{F}
    \]
    and the inclusion induced by a splitting
    \[
    T G_0 = \mathcal{F} \oplus \nu.
    \]
    
    While we don't impose any condition on $\nu$ itself at this point, however, we require for the next step that $\nu$ is retained by the connection $\mathcal{D}$ chosen for the construction of the adjoint representation up to homotopy. This can be done using different possible constructions. For example, we can choose an arbitrary $\nu$ and $\mathcal{D}'$ and then explicitly construct $\mathcal{D}$ by adding to the vectors of $\mathcal{D}'$ the terms expressing the respective failure of $\mathcal{D}'$ to retain $\nu$. We may alternatively use the notion of a Riemannian groupoid as in \cite{dHF18} using that the groupoid is proper to construct $\nu$ and $\mathcal{D}$ simultaneously, choosing $\nu$ to be specifically the complement to $\mathcal{F}$.    
    
    If we now additionally assume that $\mathcal{D}$ is furthermore contained in $E$ (we can assure this by projecting all vectors of $\mathcal{D}$ to their $E$-horizontal counterparts, which retains the action on $T G_0$ and in particular $\nu$), we can explicitly check that
    \[
    \Ad = B \oplus \left( \mathcal{F}[2] \xrightarrow{\id} \mathcal{F}[1] \right)
    \]
    in analogy to Lemma \ref{LemmaAdjRuthSplitsforMEC} and Lemma \ref{LemmaProjToNormalRepFromExtension}. As the associated inclusion, projection and splitting homotopy (which is induced by reversing the identity arrow between the copies of $\mathcal{F}$) between $C(G, S^p(B^*))$ and $C(G, S^p(\Ad^*))$ are not only compatible with the differential of the graded vector spaces, but also the differential of the representation up to homotopy, this data already describes a solution of the transference problem (alternatively we can check this explicitly, by spelling out the initiator $\delta=D-D^0$ which commutes with the splitting homotopy and obtain the same). Therefore, the isomorphism
    \[
    H(G, S^p(\Ad^*)) \cong H(G, S^p(B^*))
    \]
    is indeed direcly induced by the inclusion $B^* \subset \Ad^*$. As the elements
    \[
        \alpha_1 \in (\Gamma(G_0, \wedge^{a}(\nu^*) \otimes S^b(A^*_{\iso})))^G
    \]
    we previously considered are in the image of this inclusion, we can represent the cohomology classes by the same elements as before and obtain the statement of the theorem.
    
    Lastly, we may note that that this isomorphism does not depend on the choice of inclusions and projections between $B=\mathcal{H}(\Ad)$ and $\Ad$, as long as these correspond to the identity on the cohomology groups. While a priori, we may obtain different inclusions of $C(G, S^p(B^*))$ into $C(G, S^p(\Ad^*))$, the homotopy perturbation lemma provides explicit maps describing the relation between the structures. In particular, for any element in the image of some inclusion, its projection to $\Gamma(G_0, S^p(B^*))$ with respect to any structure corresponds to a unique expression. (This can be seen as the initiator $\delta= D- D^0$ always increases the degree of $G_\bullet$.) So the cohomology class $[\alpha_{-1}]$ must correspond to the expression of $\alpha_{-1}$ and therefore $\pm \alpha$ independent of the choice.
\end{proof}

\begin{cor}
    Let $G$ be a proper and regular Lie groupoid. The differential on the first page of the spectral sequence of Theorem \ref{ThmBottSSAriasAbad} can be computed by $\dCart$.
\end{cor}

\begin{proof}
    We know by Proposition \ref{PrpThesisDefisDCart}, that $\dCart$ describes the differential on the first page of the spectral sequence associated to the filtration by columns associated to the isomorphism
    \[
    H(\Omega^p(G_\bullet), \delta) \cong (\Gamma(G_0, \wedge (\nu^*) \otimes S(A^*_{\iso}) ))^G
    \]
    of Theorem \ref{ThmThesisIsoCartInclusionToColumns}. As by Theorem \ref{ThmMainTwoIsosAgree} this isomorphism agrees up to a factor of $\pm 1$ with the isomorphism used in Corollary \ref{CorColumnsViaRedRuth}, the expression for the first page differential must also transfer. The scaling factor of $\pm 1$ furthermore only depends on the degree of the vector space in this complex, corresponding to the vertical degree in the Bott-Shulman-Stasheff bicomplex. The corresponding vertical degree is retained by the differential on the first page and therefore the factors on domain and target cancel out.
\end{proof}

\section{Outlook: The Cartan model}

While it is an important first step to be able to compute the differential of the first page of the spectral sequence, an even stronger result would be the existence of a lift of this differential to a map (or even a differential) between the different representations up to homotopy $C(G, S^\bullet(\Ad^*))$ increasing the degree of the symmetric power by 1. The Cartan model can be obtained using representations up to homotopy, where there is a clear candidate for such a lift.

\begin{defi}[cf. \cite{Bl16}, p. 539]
Let $G$ be a Lie groupoid. A connection $\mathcal{D}$ with respect to $s$ is called a Cartan connection if the map
    \begin{center}
    \begin{tikzcd}
	\mathcal{D} \arrow[d, shift right, "Tt" left] \arrow[d, shift left, "Ts" right] \arrow[r, hook] & T G_1 \arrow[d, shift right, "Tt" left] \arrow[d, shift left, "Ts" right] \\
	T G_0 \arrow[r] & T G_0
	\end{tikzcd}
    \end{center}
    is an inclusion of Lie groupoids.
\end{defi}

\begin{remark}
    In the source, the condition of multiplicativity is formulated as the fact that the inclusion $G \rightarrow J^1 G$ induced by $\mathcal{D}$ needs to be a morphism of Lie groupoids. However, using the isomorphism $\mathcal{D} \cong s^*(T G_0)$, we have that the fat groupoids associated to the above VB-groupoids are precisely given by $G$ and $J^1 G$. It can easily be checked that the conditions for the morphisms to be Lie groupoid morphisms agree.
\end{remark}

\begin{prp} \label{PrpCartanConnCoadjRuth}
    Consider $G$ a Lie groupoid with a Cartan connection $\mathcal{D}$. If we construct the adjoint representation up to homotopy using $\mathcal{D}$, then we obtain a strict representation up to homotopy. In particular, we obtain that
    \[
        H(G, S^p(\Ad^*)) \cong H \left( \ldots \rightarrow \Gamma(G_0, (S^p(\Ad^*))^0)^G \rightarrow \Gamma(G_0, (S^p(\Ad^*))^1)^G \rightarrow \ldots \right)
    \]
    is an isomorphism of cohomology groups.
\end{prp}

\begin{proof}
    We need to check that the maps of Definition \ref{DefiAdjointRuth} define actions of $G$ on $A$ and $T G_0$ respectively, i.e. that associativity holds. However, the Cartan property precisely implies that for $W \in \mathcal{D}_x$ and $V \in \mathcal{D}_y$ the product satisfies $T \mu(V,W) \in \mathcal{D}_{\mu(x,y)}$, and therefore precisely the properties of $Z$ in the definition of $R_2$. Therefore, $R_2=0$, which corresponds to associativity of $R_1$. The consequence is implied by Lemma \ref{LemmaCohomStrictRuths}.
\end{proof}

\begin{example} \label{ExCartModelRecovery}
    Consider the action of a group $\mathcal{G}$ on a manifold $M$ and assume that the action is proper (e.g. whenever $\mathcal{G}$ is compact). Then the associated action groupoid comes with a product structure on $G_1= \mathcal{G} \times M$ inducing a flat Cartan connection. By Proposition \ref{PrpCartanConnCoadjRuth}, we obtain that
    \[
    H(G, S^p(\Ad^*))=H \left( \bigoplus_{a+b=p} \Gamma(G_0, \wedge^a T^* G_0 \otimes S^b(A^*))^G, d \right).
    \]
    Here, the differential is induced by $\rho^*$. Using the identification $A = \mathfrak{g} \times M$, we can write
    \[
        (\Gamma(M, \wedge^a T^* M) \otimes S^b(\mathfrak{g}^*))^{\mathcal{G}} = (\Omega^a(M) \otimes S^b(\mathfrak{g}^*))^{\mathcal{G}}
    \]
    We may note that each element described this way corresponds to a $\mathcal{G}$-invariant form in $\Omega^p(G_b)$.
\end{example}

    The spaces we see in Example \ref{ExCartModelRecovery} are those of the classical Cartan model and we would hope that the differential $d$ is related to the Cartan differential $\dCart$. By construction $d$ is induced by the anchor and we can explicitly check that it corresponds to the component of $\dCart$ increasing the degree in $S^\bullet(A^*)$. The other component of $\dCart$ is simply the de-Rham differential, applied to the first component of
    \[
            \Gamma(G_0, \wedge^a T^* G_0 \otimes S^b(A^*)) \cong \Omega^a(G_0) \otimes S^b(\mathfrak{g}^*).
    \]

    We may now use techniques similar to Theorem \ref{ThmMainTwoIsosAgree} to analyse the first page differential and how it relates to the de-Rham differential. This is currently work in progress.

\begin{conj}
    Let $G$ be a proper Lie groupoid with a Cartan connection $\mathcal{D}$. Then the differential on the first page of the spectral sequence can be lifted to a differential on
    \[
        (\Gamma(G_0, \wedge^a T^* G_0 \otimes S^b(A^*)))^G.
    \]
    It can be described by the map on the assoicated invariant differential forms given by the component of $\ddR$ retaining the vertical degree with respect to the splitting of $\mathcal{D}$. Furthermore, it may be computed by restricting the horizontal differential that we obtain from the bigrading of the Mehta generalisation of the Cartan model of \cite{Mehta06} (cf. Definition 4.2.21.).
\end{conj}


Considering the setting of a proper action groupoid, the conjecture would imply that the missing component of the Cartan differential in Example \ref{ExCartModelRecovery} indeed corresponds to a lift of the differential on the first page of the spectral sequence. However, it is not implied that the differential of the Cartan model restricts to the differential of the isotropy Cartan model, i.e. that the inclusion on the first page of the spectral sequence can be lifted by the inclusion
\[
    \Gamma(G_0, \wedge (\nu^*) \otimes S(A^*_{\iso}))^G \subset \Gamma(G_0, \wedge (T^* G_0) \otimes S(A^*))^G
\]
induced by some multiplicative Ehresmann connection $E$, even when $E$ extends the product structure.

\begin{example}
    Consider the action of $SO(3)$ on $\mathbb{S}^2$. The isotropy group at each point is given by rotation around its corresponding axis, i.e. $\mathbb{S}^1$. Also, the action is transitive, so we know that the cohomology of the Bott-Shulman-Stasheff complex has to agree with $H(B \mathbb{S}^1) = \mathbb{R}[\mu]$.

    We can now construct a multiplicative Ehresmann connection by considering the splitting
    \[
        A=\mathfrak{g} \times M=A_{\iso} \oplus A_{E}
    \]
    which at each $m \in M$ is given by splitting $\mathfrak{g}$ as $\ker(\rho_m)$ and its complement with respect to the Killing form. As $A_{E}$ is closed under the adjoint action, it extends uniquely to a left and right $\mathcal{G}$-invariant vector bundle on $\mathcal{G} \times M$ that is vertical with respect to $s$ and extends the product connection to a complement of $\ker(Tt) \cap \ker(Ts)$. This distribution is multiplicative by construction, contains the unit section and is retained by the inverse map, so it defines a multiplicative Ehresmann connection.
    
    If we identify $\mathfrak{g}$ with the skew-symmetric 3x3-matrices, the map
    \[
    \begin{pmatrix}
        x \\ y \\ z
    \end{pmatrix}
    \mapsto
    \begin{pmatrix}
        0 & z & -y \\
        -z & 0 & x \\
        y & -x & 0
    \end{pmatrix}
    \]
    trivialises $A_{\iso}$. It can be explicitly checked that the product connection on $A$ projects to a connection on $A_{\iso}$ under which this section is horizontal. As the product connection comes from the product connection on $\mathcal{G} \times M$, which is contained in the multiplicative Ehresmann connection, we find that differentiating the projection of $E$ to the isotropy groupoid must yield exactly the same connection.

    We can now choose the dual trivialisation of $A_{\iso}^*$, induced by the dual section $\gamma \in \Gamma(G_0, A^*_{\iso})$. As the section to $A_{\iso}$ chosen above can be explicitly checked to be invariant under the adjoint action, its dual is invariant under the coadjoint action. Therefore, $\gamma$ describes an isomorphism
    \[
        \Gamma(G_0, S(A^*_{\iso}))^{G} \cong S(\mathfrak{g}^*_{\iso})
    \]
    where $\mathfrak{g}_{\iso} = \mathbb{R}$ denotes the isotropy Lie algebra at an arbitrary point $m \in M$. Note that the invariance condition does not appear anymore, as the isotropy is $\mathbb{S}^1$ and hence abelian (which is also implicitly required for the existence of a trivialisation of $A_{\iso}$ induced by invariant non-vanishing sections). Furthermore, the fact that $\gamma$ was chosen dual to a horizontal section means that it vanishes under the differential of the isotropy Cartan model.

    Now shifting our focus to the product connection on $\mathfrak{g} \times M$, we can explicitly trivialise $A$ by chosing a horizontal basis corresponding to the constant sections described by the elements
    \[
        \theta^a \in \mathfrak{g}^* : \begin{pmatrix}
        0 & a & b \\
        -a & 0 & c \\
        -b & -c & 0
    \end{pmatrix}
    \mapsto a
    \]
    and $\theta^b$ and $\theta^c$ analogously defined as the maps recovering $b$ and $c$. We can now  express our distinguished section $\gamma$ as
    \[
        \gamma(x,y,z) = z \cdot \theta^a - y \cdot \theta^b + x \cdot \theta^c
    \]
    where we identify sections to $A^*$ with functions to $\mathfrak{g}^*$.

    It is immediately clear, that the differential of the Cartan model (which for degree reasons is just given by the de-Rham differential on the functions) does not vanish on $\gamma$ when naively included in the Cartan model. However, one can explicitly compute that the form
    \[
        \omega = z \cdot \ddR (x) \wedge \ddR (y) - y \cdot \ddR (x) \wedge \ddR (z) + x \cdot \ddR (y) \wedge \ddR (z)
    \]
    satisfies that $\gamma + \omega$ is closed under the Cartan differential.

    As the groupoid in our example is transitive, we have that $TG \rightarrow (\nu=0)$ is a canonical projection. This, together with the construction of the differential of the isotropy Cartan model as the exteriour covariant derivative with respect to the projection of the product connection implies that the projection
    \[
        \Gamma(G_0, \wedge (T^* G_0) \otimes S(A^*))^G \rightarrow \Gamma(G_0, S(A^*_{\iso}))^G \cong S(\mathfrak{g}^*_{\iso})
    \]
    is a map of chain complexes, where $\gamma + \omega$ is a preimage of $\gamma$. At the same time, we have that
    \[
        \mathbb{R} \cong \mathfrak{g}^* \rightarrow \Gamma(G_0, \wedge (T^* G_0) \otimes S(A^*))^G
    \]
    linearly extending mapping $1$ to $\gamma + \omega$ yields an inclusion
    \[
        S(\mathfrak{g}^*_{\iso}) \rightarrow \Gamma(G_0, \wedge (T^* G_0) \otimes S(A^*))^G.
    \]
    These maps now identify the isotropy Cartan model as a quasi-isomorphic subcomplex of the Cartan model with a canonical projection. (The fact that the maps induce an isomorphism on cohomology follows for example from the fact that the projection map agrees with the pull-back map to the single point $m \in M$ of which we already know that it is a quasi-isomorphism.) It is notable that the expressions $\gamma$ and $\gamma + \omega$ generating the cohomology are closely related.
\end{example}

So while the differentials and therefore the closed forms of the isotropy Cartan model and Cartan model in this case do not agree, we can reasonably expect a close relation between the two approaches.

\printbibliography

\end{document}